\documentclass[reqno,12pt,letterpaper]{amsart}
\usepackage{amsmath,amssymb,amsthm,graphicx,mathrsfs,url}
\usepackage[usenames,dvipsnames]{color}
\usepackage[colorlinks=true,linkcolor=Red,citecolor=Green]{hyperref}
\usepackage[usenames,dvipsnames]{xcolor}
\usepackage{tikz}
\usepackage{graphicx}
\usepackage[colorlinks=true]{hyperref}
\usepackage{color}
\usepackage{amsmath,amsfonts}
\usepackage{calrsfs}
\usepackage{amsmath,environ}
\DeclareMathAlphabet{\pazocal}{OMS}{zplm}{m}{n}

\numberwithin{equation}{section}
\usepackage{amsmath, amsthm}
    \usepackage{dsfont}
    \usepackage{fancybox}
    \usepackage{amssymb}

\newcommand{\R}{\mathbb{R}}

\newcommand{\p}{{\partial}}

\newcommand{\dd}[2]{\dfrac{\partial #1}{\partial #2}}

\newcommand{\Det}{{\operatorname{Det}}}

\newcommand{\DD}{{\mathcal{D}}}

\newcommand{\epsi}{\varepsilon}

\newcommand{\trace}{{\operatorname{Tr}}}

\newcommand{\lr}[1]{\langle #1 \rangle}

\newcommand{\supp}{\mathrm{supp}}

\newcommand{\Bb}{\mathbb{B}}
\newcommand{\Z}{\mathbb{Z}}

\newcommand{\Id}{{\operatorname{Id}}}
\newcommand{\BB}{\mathcal{B}}
\newcommand{\VV}{{\mathcal{V}}}

\newcommand{\LL}{{\mathcal{L}}}

\newcommand{\C}{\mathbb{C}}
\newcommand{\oz}{\overline{z}}

\newcommand{\az}{\alpha}

\renewcommand{\Re}{\operatorname{Re}}

\newcommand{\Tt}{{\mathbb{T}}}

\newcommand{\lamdba}{{\lambda}}

\newcommand{\de}{\mathrel{\stackrel{\makebox[0pt]{\mbox{\normalfont\tiny def}}}{=}}}

\title[Bound states for rapidly oscillatory Schr\"odinger operators]{Bound states for rapidly oscillatory Schr\"odinger operators in dimension $2$}
\date{\today}

\author{Alexis Drouot}
\email{alexis.drouot@gmail.com}

\NewEnviron{equations}{%
\begin{equation}\begin{gathered}
  \BODY
\end{gathered}\end{equation}
}

\NewEnviron{equations*}{%
\begin{equation*}\begin{gathered}
  \BODY
\end{gathered}\end{equation*}
}

\newtheorem{thm}{Theorem}

\newtheorem{lem}{Lemma}[section]

\newtheorem{theorem}[thm]{Theorem}
\theoremstyle{definition}

\begin{document}
\maketitle

\begin{abstract}
We study the eigenvalues of Schr\"odinger operators $-\Delta_{\R^2} + V_\epsi$ on $\R^2$ with rapidly oscillatory potential $V_\epsi(x) = W(x,x/\epsi)$, where $W(x,y) \in C^\infty_0(\R^2 \times \Tt^2)$ satisfies $\int_{\Tt^2} W(x,y) dy = 0$. We show that for $\epsi$ small enough, such operators have a unique negative eigenvalue, that is exponentially close to $0$.
\end{abstract}

\section{Introduction.}

We study the $L^2$-eigenvalues of the Schr\"odinger operator $-\Delta_{\R^2}+V_\epsi$ on $\R^2$, where $\epsi$ is a small parameter and $V_\epsi$ is a real-valued, compactly supported, and rapidly oscillatory potential on $\R^2$:
\begin{equation*}
V_\epsi(x) = W(x,x/\epsi), \ \ W(x,y) = \sum_{k \in \Z^2 \setminus 0} W_k(x) e^{iky}.
\end{equation*}
The functions $W_k$ are smooth, with support in $\Bb(0,L) = \{ x \in \R^2, |x| < L \}$, and satisfy $\overline{W_k}=W_{-k}$. The potential $V_\epsi$ is a first approximation to model disordered medias with scale of heterogeneity $\sim \epsi$ -- oscillations play here the role of randomness. This note provides a simple proof of a conjecture of Duch\^ene--Vuki\'cevi\'c--Weinstein \cite{DVW}:

\begin{theorem}\label{thm:1} For $\epsi$ small enough, the operator $-\Delta_{\R^2}+V_\epsi$ has a unique bound state, with energy $E_\epsi$ given by
\begin{equation}\label{eq:1t}
E_\epsi  = - \exp\left( - \dfrac{4\pi}{\epsi^2 \int_{\R^2} \Lambda_0(x) dx + o(\epsi^2) } \right), \ \ \ \Lambda_0(x) \de \sum_{k \neq 0} \dfrac{|W_k(x)|^2}{|k|^2}.
\end{equation}
\end{theorem}

In one dimension, the study of the spectral properties of $-\p_x^2+V_\epsi$ with $V_\epsi$ rapidly oscillatory originated with Borisov--Gadyl'shin \cite{BG}, who gave a sufficient condition for the existence of a bound state, and derived the asymptotic of the corresponding energy:
\begin{equation}\label{eq:1o}
E_\epsi = -\dfrac{\epsi^4}{4} \int_{\R} \Lambda_0(x) dx, \ \ \Lambda_0(x) \de \sum_{k \neq 0} \dfrac{|W_k(x)|^2}{|k|^2}.
\end{equation}
This study was continued later in Borisov \cite{B} then Duch\^ene--Weinstein \cite{DW}, in particular to include less regular potentials.

Duch\^ene--Vuki\'cevi\'c--Weinstein \cite{DVW} studied  the behavior of scattering quantities of Schr\"odinger operators with potentials which are the sum of a slowly varying term $W_0$ and a rapidly oscillatory term $V_\epsi$. They showed that the transmission coefficient of the \textit{effective potential} $W_0-\epsi^2 \Lambda_0$, which is a $O(\epsi^2)$-perturbation of $W_0$, differs from the one of $W_0+V_\epsi$ by $O(\epsi^3)$.  When $W_0 = 0$, they recovered the result of \cite{BG}: $-\p_x^2+V_\epsi$ has a unique negative eigenvalue satisfying \eqref{eq:1o} for small $\epsi$. They proved a uniform weighted dispersive estimate for the propagator $e^{it(-\p_x^2+V_\epsi)}$ as $\epsi \rightarrow 0$, despite the presence of an eigenvalue near the edge of the continuous spectrum.

Recently, Duch\^ene--Raymond \cite{DR} obtained homogenization results for potentials of the form $\epsi^{-\beta} V_\epsi$, $\beta \in (0,2)$, in dimension $1$, using a normal form approach. Dimassi \cite{Di} applied $\epsi$-semiclassical calculus to show a trace formula and a Weyl law for potentials of the form $\epsi^{-2} V_\epsi$, in any dimension $d$. 

Motivated by \cite{DVW} and by Christiansen \cite{Ch06}, we used a different approach to study in \cite{D} the resonances and eigenvalues of $-\Delta_{\R^d}+W_0+V_\epsi$, in any odd dimension $d$. When $W_0 = 0$, we proved that the resonances and eigenvalues of $V_\epsi$ escape all bounded regions as $\epsi \rightarrow 0$ (except the one converging to $0$ when $d=1$). When $W_0 \neq 0$, we showed that the resonances and eigenvalues of $W_0+V_\epsi$ converge to the one of $W_0$, with a complete expansion in powers of $\epsi$. We improved upon the homogeneization result of \cite{DVW}, refining the effective potential $W_0-\epsi^2 \Lambda_0$ to $W_0-\epsi^2 \Lambda_0-\epsi^3 \Lambda_1$, and deriving it for any odd $d$. We refer to \cite[Figure 2]{D} for numerical results and to \cite[\S 1]{D} for additional references.

A famous result of Simon \cite[Theorem 3.4]{Si} in dimension $2$ predicts that under suitable conditions, a Schr\"odinger operator with small negative potential $-\Delta_{\R^2} + \epsilon\Lambda$ has a unique bound state with energy 
\begin{equation}\label{eq:1f}
E_\epsi  = -\exp\left( -\dfrac{4\pi}{\epsilon \int_{\R^2} \Lambda(x) dx+ o(\epsilon)}   \right).
\end{equation}
This identity put together with \eqref{eq:1t} supports the main idea of \cite{DVW}: the scattering quantities of rapidly oscillatory potentials are similar to the one of suitable small potentials. 

To prove Theorem \ref{thm:1}, we first follow \cite[\S 3]{Si}: we use a modified Fredholm determinant to reduce the study of eigenvalues of $-\Delta_{\R^2}+V_\epsi$ to an equation involving a certain trace. Then, we provide estimates on this trace following some ideas of \cite{D}, ending the proof. We mention that the result of Theorem \ref{thm:1} still applies when $W$ is not smooth but satisfies instead the weaker bound
\begin{equation*}
\sum_{k \neq 0} |W_k|_\infty + \dfrac{|W_k|_{C^1}}{|k|}+ \dfrac{|W_k|_{C^2}}{|k|^2}+ \dfrac{|W_k|_{C^3}}{|k|^3} + \sum_{0 \neq k \neq \ell} \dfrac{|W_k|_{C^3}|W_\ell|_{C^3}}{|k-\ell|^{5/2}} < \infty,
\end{equation*}
with no change in the proof.

\noindent \textbf{Aknowledgement.} We would like to thanks Maciej Zworski for valuable discussions. This research was supported by the NSF grant DMS-1500852 and the Fondation CFM pour la recherche.

\subsection*{Notations} We will use the following:
\begin{itemize}
\item For $x\in \R^2$, $\lr{x}$ denotes the Japanese bracket of $x$: $\lr{x} \de (1+|x|^2)^{1/2}$.
\item The function $z \mapsto \ln(z)$ denotes the holomorphic logarithm on $\C \setminus (-\infty,0]$.
\item $C_0^\infty(\R^2,\R)$ is the set of real smooth compactly supported functions on $\R^2$.
\item $L^2$ denotes the Hilbert space of functions that are squared integrable, with $L^2$-norm denoted by $|f|_2$ and scalar product given by $\lr{f, g}_2 = \int_{\R^2} \overline{f} g$.
\item $C^k$ is the space of functions on $\R^2$ with $k$ continuous and bounded derivatives, with norm $|f|_{C^k} = |f|_\infty+|f'|_\infty+...+|f^{(k)}|_\infty$.
\item We write $\Delta$ for the Laplacian $\Delta_{\R^2}$, and $\hat{f}$ for the Fourier transform of $f$: $\hat{f}(\xi) = \frac{1}{2\pi} \int_{\R^2} e^{-ix\xi} f(x) dx$. $H^s$ denotes the standard Sobolev space on $\R^2$, with norm $|f|_{H^s} = |\lr{\xi}^s \hat{f}|_2 = |(\Id - \Delta)^{s/2}f|_2$.
\item $\BB$ is the Banach space of bounded linear operators from $L^2$ to itself.
\item $\LL^2$ is the Banach space of Hilbert--Schmidt operators on $L^2$.
\end{itemize}

\section{General properties.}

Let $R_0(\lambda,x,y)$ be the kernel of the free resolvent $(-\Delta + \lambda^2)^{-1}$. It its a Hankel function of $\lambda |x-y|$ and one can write
\begin{equation*}
R_0(\lambda,x,y) = - \dfrac{1}{2\pi} \ln(\lambda) + H_0(\lambda,x,y),
\end{equation*}
where for every $x \neq y$ the function $\lambda \mapsto H_0(\lambda,x,y)$ is holomorphic in $\{\Re \lambda > 0\}$ and admits a continuous extension to $\{ \Re \lambda \geq 0 \}$ -- see \cite[(12)]{Si}. For $\lambda$ with $\Re \lambda \geq 0$, let $H_0(\lambda)$ be the operator with kernel $H_0(\lambda,x,y)$. When $\lambda \in (0,\infty)$, $R_0(\lambda)$ is selfadjoint and $\ln(\lambda)$ is real, hence $H_0(\lambda)$ is selfadjoint.

Let $\VV \in C_0^\infty(\R^2,\R)$, with support in $\Bb(0,L) = \{ x \in \R^2, |x| \leq L \}$, and $\rho \in C_0^\infty(\R^2, \R)$, equal to $1$ on the support of $\VV$ and $0$ outside $\Bb(0,L)$. We will be interested in the behavior of $K_\VV(\lambda) = \rho R_0(\lambda) \VV$ for $ \lambda$ close to $0$. Write $K_\VV(\lambda) = \Pi_\VV \ln(\lambda) + L_\VV(\lambda)$,
\begin{equation*}
\Pi_\VV \de -\dfrac{1}{2\pi} \rho \otimes \VV, \ \ \ L_\VV(\lambda) \de \rho H_0(\lambda) \VV, \ \ \ \Re \lambda \geq 0.
\end{equation*}
An operator $A : C^\infty_0(\R^2) \rightarrow \DD'(\R^2)$ belongs to $\LL^2$ (the Hilbert--Schmidt class) if and only if its kernel $A(x,y)$ is in $L^2(\R^4,dxdy)$. In this case the $\LL^2$-norm is
\begin{equation*}
|A|_{\LL^2} \de \left(\int_{\R^2 \times \R^2} |A(x,y)|^2 dx dy\right)^{1/2}.
\end{equation*}
We prove the following result, similar to \cite[Proposition 3.2]{Si}:

\begin{lem}\label{lem:1b} The operator $L_\rho(\lambda)$ is in $\LL^2$. Moreover, uniformly locally in $\lambda, \mu$ with nonnegative real parts,
\begin{equation}\label{eq:1b}
|L_\rho(\lambda)-L_\rho(\mu)|_\BB = O(|\lambda \ln(\lambda)-\mu \ln(\mu)|), \ \ \ \ 
|L_\rho(\lambda)|_{L^2 \rightarrow H^2} = O(1).
\end{equation}
\end{lem}

\begin{proof} To prove the first part of \eqref{eq:1b}, we write the kernel of $H_0(\lamdba)$ as
\begin{equation*}
H_0(\lambda,x,y) = -\dfrac{1}{2\pi} \ln |x-y| + F(\lambda|x-y|).
\end{equation*}
The function $F : \C \rightarrow \C$ takes the form $F(\zeta) = \zeta h(\zeta) \ln(\zeta) + g(\zeta)$ for some entire functions $g,h$ -- see \cite[(13)]{Si}. For $x \in \supp(\rho), y \in \supp(\rho)$ and $\lambda$ in a compact set, $\lambda |x-y|$ is uniformly bounded. Hence, there exists a constant $C$ such that for such values of $x,y,\mu,\lambda$,
\begin{equation*}
|F(\lambda|x-y|) - F(\mu |x-y|)| \leq C |\lambda \ln(\lambda) - \mu \ln(\mu)| \cdot |\ln |x-y||. 
\end{equation*}
The bound $|L_\rho(\lambda)-L_\rho(\mu)|_\BB = O(|\lambda \ln(\lambda)-\mu \ln(\mu)|)$ follows now from Schur's test and from the local integrability of $\ln|x-y|$ on $\R^2$.

The operator $L_\rho(\lambda)$ belongs to $ \LL^2$ by \cite[Proposition 3.2]{Si}. The second estimate of \eqref{eq:1b} amounts to prove that $\Delta L_\rho(\lambda)$ belongs to $\BB$, uniformly locally for $\lambda \in K$. By the same argument as in \cite[Theorem 2.1]{DZ}, it suffices to show that $\rho \Delta H_0(\lambda) \rho$ belongs to $\BB$. Recall that $\Delta \ln|x| = -2\pi \delta$, so that
\begin{equation}\label{eq:1g}
\Delta_x H_0(\lambda,x,y) = \delta(x-y) + \Delta_x (F(\lambda|x-y|))
\end{equation}
We now compute the Laplacian of $F(\lambda|x-y|)$ with respect to $x$. Note that
\begin{equation}\label{eq:1w}
\Delta (g(\lambda |z|)) =  4\dd{^2 g(\lambda |z|)}{z \p \oz} = \dfrac{\lambda}{|z|} g'(\lambda |z|) + \lambda^2 g''(\lambda |z|).
\end{equation}
Define $f(\lambda |z|) \de \lambda |z| h(\lambda |z|)$. By applying \eqref{eq:1w} to $f \cdot \ln$ instead of $g$, we get \begin{equations}\label{eq:1x}
\Delta (f(\lambda |z|) \ln(\lambda|z|)) = -2\pi f(\lambda |z|) \delta + \dfrac{2 f'(\lambda |z|)}{\lambda |z|} + \left(\dfrac{\lambda f'(\lambda |z|)}{|z|}  + \lambda^2 f''(\lambda |z|)\right) \ln(\lambda |z|) \\
= \dfrac{2f'(\lambda |z|)}{|z|}  + \left(\dfrac{\lambda f'(\lambda |z|)}{|z|}  + \lambda^2 f''(\lambda |z|)\right) \ln(\lambda |z|).
\end{equations} 
In the second line we used that the product of a smooth function vanishing at $0$ with $\delta$ vanishes. The right hand sides of \eqref{eq:1w} and \eqref{eq:1x} both define locally integrable functions of $z \in \C$, uniformly locally in $\lambda$. Since $\rho$ is compactly supported, Schur's test combined with \eqref{eq:1g} shows that $\rho \Delta H_0(\lambda) \rho$ belongs to $\BB$ uniformly in $\lambda$. This concludes the proof. \end{proof}

Since $L_\VV(\lambda)$ is in $\LL^2$, we can define the modified Fredholm determinant $d_\VV(\lambda)$ by
\begin{equation*}
d_\VV(\lambda) \de \Det(\Id + \Psi(L_\VV(\lambda))), \ \ \Psi(z) \de (1+z)e^{-z}-1, \ \ \Re(\lambda) \geq 0,
\end{equation*}
see \cite[Appendix B]{DZ}. As $K_\VV(\lambda)$ is the sum of the rank one operator $\Pi_\VV\ln(\lambda)$ with $L_\VV(\lambda)$, $K_\VV(\lambda)$ belongs to $\LL^2$. We define
\begin{equation*}
D_\VV(\lambda) \de \Det(\Id + \Psi(K_\VV(\lambda))), \ \ \Re(\lambda) > 0.
\end{equation*}
The negative eigenvalues of $-\Delta + \VV$ are exactly the numbers of the form $-\lambda^2$, where  $\lambda \in (0,\infty)$ is a zero of $D_\VV(\lambda)$, see \cite[Theorem 5.4]{GLMZ05}. The next result is similar to \cite[Theorem 3.3]{Si}.

\begin{lem}\label{lem:1c} Let $\lambda$ such that $\Re \lambda > 0$ and $\Id + L_V(\lambda)$ is invertible on $L^2$. Then,
\begin{equation*}
D_\VV(\lambda) = 0 \ \Leftrightarrow \ 1+ \ln(\lambda) \varphi_\VV(\lambda) = 0, \ \ \ \ \varphi_\VV(\lambda) \de \trace\left((\Id + L_\VV(\lambda))^{-1} \Pi_\VV\right).
\end{equation*}
If moreover $\lambda \in (0,\infty)$ then $\varphi_\VV(\lambda) \in \R$.
\end{lem}

\begin{proof} To simplify the notations of this proof, we simply write $L_\VV, K_\VV$ for the operators $L_\VV(\lambda), K_\VV(\lambda)$. Let $\lambda$ with $\Re \lambda > 0$ and $\Id + L_\VV$ invertible on $L^2$. Since $K_\VV =  \Pi_\VV\ln(\lambda) + L_\VV$ we have
\begin{equation}\label{eq:1h}
\Id + K_\VV=(\Id + L_\VV) \left( \Id + (\Id + L_\VV)^{-1} \Pi_\VV \ln(\lambda) \right).
\end{equation}
Recall that $1+z = (1+\Psi(z)) e^z$ so that for every bounded operator $A$, $\Id + A = (\Id + \Psi(A)) e^A$. This identity combined with \eqref{eq:1h} shows that
\begin{equation}\label{eq:1i}
e^{K_\VV}\left(\Id + \Psi(K_\VV) \right) =e^{L_\VV}(\Id + \Psi(L_\VV)) \left( \Id + (\Id + L_\VV)^{-1} \Pi_\VV \ln(\lambda) \right).
\end{equation}
We observe that $K_\VV-L_\VV = \ln(\lambda) \Pi_\VV$. Since $\Pi_\VV$ has kernel $-\rho \otimes \VV/(2\pi)$, its trace is equal to $\az = -\int_{\R^2} \VV/(2\pi)$. Taking the determinant on both sides of \eqref{eq:1i} we obtain
\begin{equation*}
\lambda^\az D_\VV(\lambda) = d_\VV(\lambda) \cdot \Det\left( \Id + (\Id + L_\VV)^{-1} \Pi_\VV \ln(\lambda) \right).
\end{equation*}
Since the operator $(\Id + L_\VV)^{-1} \Pi_\VV \ln(\lambda)$ is of rank one, the determinant of $\Id + (\Id + L_\VV)^{-1} \Pi_\VV \ln(\lambda)$ is equal to $1+\trace\left((\Id + L_\VV)^{-1} \Pi_\VV \ln(\lambda) \right)$. Defining $\varphi_\VV(\lambda) = \trace\left((\Id + L_\VV)^{-1} \Pi_\VV\right)$, we obtain
\begin{equation}\label{eq:1u}
\lambda^\az D_\VV(\lambda) = d_\VV(\lambda) \cdot \left( 1+\ln(\lambda) \varphi_\VV(\lambda) \right).
\end{equation}
Since $\Re \lambda > 0$ we have $\lambda \neq 0$ and the first part of the lemma follows.

For the second part of the lemma, it suffices to prove that $D_\VV(\lambda)$ and $d_\VV(\lambda)$ are both real when $\lambda > 0$, thanks to \eqref{eq:1u}. For such $\lambda$, $R_0(\lambda)$ is selfadjoint. Hence,
\begin{equation*}
\overline{D_\VV(\lambda)} = \Det(\Id + \Psi(K_\VV^*)) = \Det(\Id + \Psi(\VV R_0(\lambda) \rho)). 
\end{equation*}
Since $\Psi$ vanishes at $0$, there exists $\psi$ entire with $\Psi(z) = z\psi(z)$. Therefore,
\begin{equation}\label{eq:1v}
\overline{D_\VV(\lambda)} = \Det(\Id + \VV R_0(\lambda) \rho \psi(\VV R_0(\lambda) \rho)) = \Det(\Id + \rho R_0(\lambda) \rho \cdot \psi(\VV R_0(\lambda) \rho) \VV).
\end{equation}
In the last equality, we used that if $B \in \BB$ and $A$ is trace-class then $\Det(\Id + BA) = \Det(\Id + AB)$, see \cite[(B.5.13)]{DZ}. The power series expansion of $\psi(z) = \sum_{m=0}^\infty a_m z^m$ implies that $\rho R_0(\lambda) \rho \cdot \psi(\VV R_0(\lambda) \rho) \VV$
\begin{equation*}
 = \rho R_0(\lambda) \rho \sum_{m=0}^\infty a_m (\VV R_0(\lambda) \rho)^m \VV = \sum_{m=0}^\infty a_m (\rho R_0(\lambda) \VV)^{m+1} = \Psi(L_\VV).
\end{equation*}
Equation \eqref{eq:1v} now shows that $D_\VV(\lambda) = \overline{D_\VV(\lambda)}$ when $\lambda \in (0,\infty)$. The same arguments (using that $H_0(\lambda)$ is selfadjoint for $\lambda \in (0,\infty)$) shows that $d_\VV(\lambda) \in \R$ if $\lambda \in (0,\infty)$. Hence, \eqref{eq:1u} shows that $\varphi_\VV(\lambda) \in \R$ when $\lamdba \in (0,\infty)$.
\end{proof}

\section{Bound state exponentially close to zero.}

We now focus on the case of a potential $V = V_\epsi$ given by
\begin{equation*}
V_\epsi(x) \de W\left(x,\dfrac{x}{\epsi} \right), \ \ \ \
W(x,y) \de \sum_{k \in \Z^2 \setminus 0} W_k(x)e^{iky}.
\end{equation*}
Here the functions $W_k \in C^\infty(\R^2)$ have supports in $\Bb(0,L) = \{ x \in \R^2, |x| < L \}$, and $\overline{W_k} = W_{-k}$. To simplify notations, we will drop the index $\epsi$ in $V_\epsi$ and write $V = V_\epsi$. We first investigate the invertibility of $\Id + L_V(\lambda)$ for $\epsi$ small enough. We will need the following lemma regarding the behavior as $\epsi \rightarrow 0$ of certain oscillatory integrals:

\begin{lem}\label{lem:1a} The following estimates hold:
\begin{equation*}
|V|_{H^{-2}} = O(\epsi^2), \ \ \ \ \ \left| \lr{D}^{-2} V \lr{D}^{-2} \right|_\BB = O(\epsi^2), \ \ \lr{D} \de (\Id - \Delta)^{1/2}.
\end{equation*}
\end{lem}

\begin{proof} We estimate $|V|_{H^{-2}}$ in a similar way as in the proof of \cite[Theorem 1]{D}:
\begin{equations*}
|V|_{H^{-2}}    = \left| \lr{\xi}^{-2} \sum_{k \neq 0} \widehat{W_k}(\xi -k/\epsi) \right|_2 
    \leq \sum_{k \neq 0} \left| \lr{\xi}^{-2}\lr{\xi -k/\epsi}^{-2}\lr{\xi -k/\epsi}^2\widehat{W_k}(\xi -k/\epsi) \right|_2 \\ \leq \sum_{k \neq 0} |\lr{\xi}^{-2}\lr{\xi -k/\epsi}^{-2}|_\infty \cdot |W_k|_{H^2} 
     \leq C\sum_{k \neq 0} \lr{k/\epsi}^{-2} |W_k|_{H^2} = O(\epsi^2).
\end{equations*}
In the last line we used Peetre's inequality: for $t > 0$, $x,y \in \R^2$,
\begin{equation}\label{eq:1m}
\lr{x}^{-t} \lr{y}^{-t} \leq 2^t \lr{x-y}^{-t}.
\end{equation}
This shows the first estimate.

The boundedness of $\lr{D}^{-2} V \lr{D}^{-2}$ on $L^2$ is equivalent to the boundedness of the multiplication operator by $V$ from $H^2$ to $H^{-2}$. We recall that $H^2$ is an algebra in dimension $2$, and that there exists $C > 0$ such that for any $u,f \in H^2$, $|fu|_{H^2} \leq C|f|_{H^2} |u|_{H^2}$. The multiplication operator by $u \in H^2$ is bounded from $H^2$ to $H^2$, hence from $H^{-2}$ to $H^{-2}$ with same norm and we deduce the inequality $|fu|_{H^{-2}} \leq C|f|_{H^{-2}} |u|_{H^2}$. Applying this with $V=f$ shows that $|V|_{H^2 \rightarrow H^{-2}} \leq C |V|_{H^{-2}} = O(\epsi^2)$. This ends the proof
 \end{proof}

\begin{lem}\label{lem:1d} For every compact subset $K$ of $\{\lambda : \Re(\lambda) \geq 0\}$, there exist $C,\epsi_0 > 0$ such that for all $0 < \epsi \leq \epsi_0$, the inverse of $\Id + L_V(\lambda)$ exist, is given by a convergent Neumann series, and satisfies
\begin{equation*}
\left|(\Id + L_V(\lambda))^{-1}\right|_\BB \leq C.
\end{equation*}
\end{lem}

\begin{proof} To prove the lemma it suffices to show that $L_V(\lambda)^2$ is bounded with $|L_V(\lambda)^2|_\BB < \frac{1}{2}$, for $\epsi$ small enough. Recall that $V \in L^\infty$; and (Lemma \ref{lem:1b}) that for $\lambda \in K$, $L_\rho(\lambda)$ is uniformly bounded from $L^2$ to $H^2$ (hence by the adjoint bound, from $H^{-2}$ to $L^2$). These facts imply
\begin{equation*}
|L_V(\lambda)^2|_\BB \leq C|L_\rho(\lambda) V L_\rho(\lambda)|_\BB \leq C |\lr{D}^{-2}V\lr{D}^{-2}|_\BB. 
\end{equation*}
The RHS is $O(\epsi^2)$ by Lemma \ref{lem:1a}. This ends the proof.\end{proof}

The negative eigenvalues $-\lambda^2$ of the operator $-\Delta + V$ on $L^2$ all belong to a fixed compact set: if $-\lambda^2$ is a negative eigenvalue, then there exists $0 \neq u \in L^2$ such that 
\begin{equation*}
-\Delta u + V u + \lambda^2 u = 0.
\end{equation*}
Multiplying by $\overline{u}$ on both sides and integrating by parts, we obtain 
\begin{equation*}
|\nabla u|_2^2 + \lr{Vu, u}_2 + \lambda^2 |u|_2^2 = 0.
\end{equation*}
Thus, $\lambda^2 |u|_2^2 \leq |V|_\infty |u|_2^2$, which shows that $-\lambda^2 \in [-|V|_\infty, 0] \subset [-M^2,0]$ where $M = 2+\sum_{k \neq 0} |W_k|_\infty$ is independent of $\epsi$.

This fact, together with Lemma \ref{lem:1c} and \ref{lem:1d} (applied with $K = K_0 \de [0, M]$), implies that for $\epsi$ small enough the set of negative eigenvalues of $-\Delta + V$ is exactly $\{-\lambda^2\}$, where $\lambda$ solves the equation
\begin{equation*}
1+ \ln(\lambda) \varphi_V(\lambda) = 0, \ \ \lambda \in K_0, \ \ \varphi_V(\lambda) \de \trace\left((\Id + L_V(\lambda))^{-1} \Pi_V\right).
\end{equation*}

\begin{lem}\label{lem:1f} Uniformly for $\lambda, \mu$ in $K_0 = [0,M]$,
\begin{equations}\label{eq:1e}
|\varphi_V(\lambda)| = O(\epsi^2), \ \ \ \ |\varphi_V(\lambda)-\varphi_V(\mu)| = O(|\lambda \ln(\lambda) - \mu\ln(\mu)|).
\end{equations}
\end{lem}

\begin{proof} We work exclusively with $\lambda, \mu \in K_0$ and $\epsi$ small enough so that Lemma \ref{lem:1d} applies. All the estimates below are uniform for such $\lambda, \epsi$. Write $(\Id + L_V(\lambda))^{-1} = \Id - L_V(\lambda) (\Id + L_V(\lambda))^{-1}$ to get
\begin{equations}\label{eq:1r}
\varphi_V(\lambda)  =  \trace(\Pi_V) - \trace(L_V(\lambda) (\Id + L_V(\lambda))^{-1} \Pi_V) = \dfrac{-\lr{\rho, V}_2 + \lr{L_V(\lambda) f_\lambda, V}_2}{2\pi}, \\
f_\lambda \de (\Id + L_V(\lambda))^{-1}\rho.
\end{equations}
The term $\lr{\rho,V}_2$ is clearly independent of $\lambda$ and equal to $O(\epsi^\infty)$. Let $P_k = (k_1 D_{x_1} + k_2 D_{x_2})/|k|^2$, so that $P_k e^{ikx/\epsi} = \epsi e^{ikx/\epsi}$. Since $P_k$ is selfadjoint, $\lr{L_V(\lambda) f_\lambda, V}_2 =$
\begin{equation*}
\epsi^2 \sum_{k \neq 0} \lr{ W_k \cdot L_V(\lambda) f_\lambda, \rho P_k^2 e^{-ik\bullet/\epsi}}_2 = \epsi^2 \sum_{k \neq 0} \lr{ P_k^2 \left(W_k \cdot L_V(\lambda) f_\lambda\right), \rho e^{-ik\bullet/\epsi}}_2.
\end{equation*}
The Cauchy--Schwarz inequality yields
\begin{equations*}
|\varphi_V(\lambda)| \leq O(\epsi^\infty) + C\epsi^2 \sum_{k \neq 0} \left|P_k^2 \left(W_k \cdot L_V(\lambda) f_\lambda\right)\right|_2.
\end{equations*}
By Lemma \ref{lem:1d}, the operator $(\Id + L_V(\lambda))^{-1}$ is uniformly bounded in $\BB$, hence $|f_\lambda|_2 = O(1)$. Since $L_V(\lambda)$ maps $L^2$ to $H^2$ the operator $P_k^2 W_k L_V(\lambda)$ maps $L^2$ to $L^2$ with
\begin{equation*}
\left|P_k^2 W_k L_V(\lambda)\right|_\BB \leq C\dfrac{|W_k|_{C^2}}{|k|^2} |V|_\infty.
\end{equation*}
Thus $|\varphi_V(\lambda)| = O(\epsi^2)$ as claimed. 

Regarding the second estimate in \eqref{eq:1e}, we use \eqref{eq:1r} and the Cauchy--Schwarz inequality to see that
\begin{equations*}
|\varphi_V(\lambda)-\varphi_V(\mu)|  \leq  \dfrac{1}{2\pi} |L_V(\lambda) f_\lambda - L_V(\mu) f_\mu|_2 |V|_2 \\
\leq C |L_\rho(\lambda)-L_\rho(\mu)|_\BB + C |f_\lambda-f_\mu|_2 \leq C |L_\rho(\lambda)-L_\rho(\mu)|_\BB.
\end{equations*}
In the above we used $f_\lambda = (\Id + L_V(\lambda))^{-1} \rho$, the resolvent identity, and the uniform boundedness of $(\Id + L_V(\lambda))^{-1}$. The conclusion follows now from Lemma \ref{lem:1b}.
\end{proof}

We now identify $\lim_{\epsi \rightarrow 0} \epsi^{-2}\varphi_V(\lambda)$:

\begin{lem}\label{lem:1e} Uniformly for $\lambda \in K_0$,
\begin{equation}\label{eq:1c}
\varphi_V(\lambda) = \dfrac{\epsi^2}{2\pi} \int_{\R^2} \Lambda_0(x) dx + o(\epsi^2), \ \ \Lambda_0(x) \de \sum_{k \neq 0} \dfrac{|W_k(x)|^2}{|k|^2}.
\end{equation}
\end{lem}

\begin{proof} 1. We claim that it is enough to show \eqref{eq:1c} for $\lambda \in [1,2]$. Indeed, Lemma \ref{lem:1f} shows that $\epsi^{-2} \varphi_V$ is uniformly bounded. In addition, it is a holomorphic function of $\lambda \in K_0$. Hence, after possibly passing to a subsequence, it converges uniformly locally to a holomorphic function. If \eqref{eq:1c} holds for $\lambda \in [1,2]$, $\epsi^{-2} \varphi_V$ has only one accumulation point on $[1,2]$, hence -- by the unique continuation principle -- only one accumulation point on $K_0$. This would prove that \eqref{eq:1c} holds uniformly on $K_0$. Below we show \eqref{eq:1c} for $\lambda \in [1,2]$. We will always assume that $\epsi$ is small enough so that Lemma \ref{lem:1d} holds.

2. By expanding $(\Id + L_V(\lambda))^{-1}$ into a finite Born series,
\begin{equations*}
\varphi_V(\lambda) = \trace(\Pi_V) - \trace(L_V(\lambda) \Pi_V) + \trace\left(L_V(\lambda) (\Id + L_V(\lambda))^{-1} L_V(\lambda) \Pi_V\right) \\
= -\dfrac{1}{2\pi} \int_{\R^2} V(x) dx + \dfrac{1}{2\pi} \lr{L_V(\lambda) \rho, V}_2 -\dfrac{1}{2\pi} \lr{(\Id + L_V(\lambda))^{-1} L_V(\lambda) \rho, L_V(\lambda)^* V}_2.
\end{equations*}
The first term is $O(\epsi^\infty)$. We note that since $\lambda \in [1,2]$, $L_V(\lambda)^* = L_V(\lambda)$ and the third term can be bounded using the Cauchy--Schwarz inequality:
\begin{equation}\label{eq:1d}
\left|\lr{(\Id + L_V(\lambda))^{-1} L_V(\lambda) \rho, L_V(\lambda) V}_2\right| \leq \left|(\Id + L_V(\lambda))^{-1}\right|_\BB |L_V(\lambda) \rho|_2 |L_V(\lambda) V|_2.
\end{equation}
We note that $L_V(\lambda) \rho = L_\rho(\lambda) V$ and we recall that $L_V(\lambda)$ is bounded from $H^{-2}$ to $L^2$. Hence Lemma \ref{lem:1a} implies that $|L_V(\lambda) \rho|_2 \leq C |V|_{H^{-2}} = O(\epsi^2)$. Similarly,
\begin{equation*}
|L_V(\lambda) V|_2 \leq C |L_\rho(\lambda) V| = O(\epsi^2)
\end{equation*} 
This proves that the RHS of \eqref{eq:1d} is $O(\epsi^4)$. The steps below are devoted to estimating the leading term in the expansion of $\varphi_V(\lambda)$: $\lr{L_V(\lambda) \rho, V}_2$. 

3. Using $L_V(\lambda) = K_V(\lambda)-\Pi_V \ln(\lambda)$, we have uniformly on $[1,2]$
\begin{equations}\label{eq:1z}
 \lr{L_V(\lambda) \rho, V}_2 
 = \int_{\R^2 \times \R^2} R_0(\lambda,x,y) V(y) V(x) dx dy + \dfrac{\ln(\lambda)}{2\pi} \left(\int_{\R^2}V(x) dx\right)^2 \\ 
= \int_{\R^2 \times \R^2} \phi_\lambda(x-y) V(y) V(x) dx dy + O(\epsi^\infty).
\end{equations}
Here the function $\phi_\lambda$ has Fourier transform equal to $\widehat{\phi_\lambda}(\xi) = (|\xi|^2+\lambda^2)^{-1}$. The identity \eqref{eq:1z} and the Plancherel formula show that $\lr{L_V(\lambda) \rho, V}_2$ is equal modulo $O(\epsi^\infty)$ to
\begin{equations*}
\lr{\widehat{\phi_\lambda} \widehat{V}, \widehat{V}}_2 = \int_{\R^2} \dfrac{\widehat{V}(\xi) \widehat{V}(-\xi)}{|\xi|^2 + \lambda^2} d\xi.
\end{equations*}
As the Fourier transform of $V$ is given by $\sum_{k \neq 0} \widehat{W_k}(\xi-k/\epsi)$ we obtain
\begin{equations*}
\lr{L_V(\lambda) \rho, V}_2  =  \sum_{k,\ell} I[W_k,W_\ell] + O(\epsi^\infty), \ \ \ \  
I[W_k,W_\ell] \de \int_{\R^2} \dfrac{\widehat{W_k}(\xi-k/\epsi) \widehat{W_\ell}(-\xi+\ell/\epsi)}{|\xi|^2+\lambda^2}  d\xi.
\end{equations*}
We will study $I[W_k,W_\ell]$ depending whether $k+\ell \neq 0$ or $k+\ell = 0$. 

4. Assume that $k+\ell \neq 0$. We have
\begin{equations*}
|I[W_k,W_\ell]| 
\leq C |W_k|_{C^3} |W_\ell|_{C^3} \int_{\R^2} \lr{\xi}^{-2} \lr{\xi-k/\epsi}^{-3} \lr{-\xi + \ell/\epsi}^{-3} d\xi \\ \leq C\epsi^{5/2} \dfrac{|W_k|_{C^3} |W_\ell|_{C^3}}{|k-\ell|^{5/2}} \int_{\R^2} \lr{\xi}^{-2} \lr{\xi-k/\epsi}^{-1/2} \lr{-\xi + \ell/\epsi}^{-1/2} d\xi \end{equations*}
To get the second line, we controlled $\lr{\xi-k/\epsi}^{-5/2} \lr{-\xi + \ell/\epsi}^{-5/2}$ by $\lr{(k-\ell)/\epsi}^{-5/2}$ with Peetre's inequality \eqref{eq:1m}. H\"older's inequality yields
\begin{equation*}
|I[W_k,W_\ell]| \leq C\epsi^{5/2} \dfrac{|W_k|_{C^3} |W_\ell|_{C^3}}{|k-\ell|^{5/2}}.
\end{equation*}

5. We now focus on the case $k+\ell=0$. A substitution yields
\begin{equation*}
I[W_k, W_{-k}] = \epsi^{2} \int_{\R^2} \dfrac{\widehat{W_k}(\xi) \widehat{W_{-k}}(-\xi)}{|\epsi\xi+k|^2+(\epsi\lambda)^2}  d\xi. 
\end{equation*}
We split the domain of integration $\R^2$ into two parts: $\Bb(0,\epsi^{-{3/4}})$ and $\R^2 \setminus \Bb(0,\epsi^{-{3/4}})$. If $\xi \in \R^2 \setminus \Bb(0,\epsi^{-{3/4}})$ then $|\epsi \xi + k| + (\epsi \lambda)^2 \geq (\epsi \lambda)^2 \geq \epsi^2$ since $\lambda \in [1,2]$. Therefore
\begin{equations*}
\left|\int_{\R^2 \setminus \Bb(0,\epsi^{-1/2})} \dfrac{\widehat{W_k}(\xi) \widehat{W_{-k}}(-\xi)}{|\epsi\xi+k|^2+(\epsi\lambda)^2}  d\xi \right| \leq \int_{|\xi| \geq \epsi^{-{3/4}}} \dfrac{1}{\epsi^2} |\widehat{W_k}(\xi) \widehat{W_{-k}}(-\xi)| d\xi \\ \leq C \int_{|\xi| \geq \epsi^{-{3/4}}} \dfrac{1}{\epsi^2} \lr{\xi}^{-6} |W_k|_{C^{3}}  |W_{-k}|_{C^{3}} d\xi  \leq C \epsi |W_k|_{C^{3}}  |W_{-k}|_{C^{3}}.
\end{equations*}

We now assume that $\xi \in \Bb(0,\epsi^{-{3/4}})$, hence $|\epsi \xi + k| \geq |k|-\epsi^{1/4} \geq |k|/2$. This implies
\begin{equations*}
\dfrac{1}{|\epsi \xi+k|^2+(\epsi \lambda)^2} -\dfrac{1}{|k|^2} = \dfrac{-(\epsi \lambda)^2 - \epsi^2 |\xi|^2 -2 \lr{\epsi\xi,k}}{|k|^2 (|\epsi \xi+k|^2+(\epsi \lambda)^2)} = O(\epsi^{1/4}|k|^{-3}).
\end{equations*}
This implies
\begin{equations*}
\int_{\Bb(0,\epsi^{-1/2})} \dfrac{\widehat{W_k}(\xi) \widehat{W_{-k}}(-\xi)}{|\epsi\xi+k|^2+(\epsi\lambda)^2} d\xi  = \int_{\Bb(0,\epsi^{-1/2})} \dfrac{\widehat{W_k}(\xi) \widehat{W_{-k}}(-\xi)}{|k|^2} d\xi + O(\epsi^{{1/4}}) \dfrac{|W_k|_{C^2} |W_{-k}|_{C^2}}{|k|^3} \\
 = \int_{\R^2} \dfrac{\widehat{W_k}(\xi) \widehat{W_{-k}}(-\xi)}{|k|^2} d\xi + O(\epsi^{1/4}) \dfrac{|W_k|_{C^2} |W_{-k}|_{C^2}}{|k|^3}.
\end{equations*}

6. Combine the results of steps 1-5 to get
\begin{equations*}
\epsi^{-2} \varphi_V(\lambda) = \dfrac{1}{2\pi} \sum_{k \neq 0} I[W_k,W_{-k}] + o(1) = \dfrac{1}{2\pi} \sum_{k \neq 0} \dfrac{1}{|k|^2} \int_{\R^2} \widehat{W_k}(\xi) \widehat{W_{-k}}(-\xi) d\xi + o(1) \\ = \dfrac{1}{2\pi} \sum_{k \neq 0} \dfrac{1}{|k|^2} \lr{W_k,\overline{W_{-k}}}_2 + o(1) = \dfrac{1}{2\pi} \int_{\R^2} \Lambda_0(x) dx+ o(1), \ \ \ \Lambda_0(x) = \sum_{k\neq 0} \dfrac{|W_k(x)|^2}{|k|^2}.
\end{equations*}
In the above we used that as $V$ is real-valued, $W_k=\overline{W_{-k}}$. This estimate completes the proof of the lemma.\end{proof}

We finally turn to the proof of Theorem \ref{thm:1}.

\begin{proof}[Proof of Theorem \ref{thm:1}] By the discussion following the proof of Lemma \ref{lem:1d}, for $\epsi$ small enough the negative eigenvalues of $-\Delta+V$  are all in a fixed  compact set and are the negative squares of the resonances of $V$ in $K_0$. We now work for $\lambda \in K_0$. 

For $\epsi$ small enough, $\Id + L_V(\lambda)$ is invertible by Lemma \ref{lem:1d}. Lemma \ref{lem:1c} shows that the resonances of $V$ are the zeros of $1+\ln(\lambda) \varphi_V(\lambda)$. Since $\varphi_V(\lambda) = \frac{\epsi^2}{2\pi}  \int_{\R^2} \Lambda_0 + o(\epsi^2)$ uniformly on $K_0 = [0,M]$, if $\epsi$ is small enough then
\begin{equation*}
\lim_{\lambda \rightarrow 0} 1+ \ln(\lambda) \varphi_V(\lambda) = -\infty, \ \ \ \  1+\ln(M) \varphi_V(M) \geq \dfrac{1}{2}.
\end{equation*}
We recall that $\varphi_V$ is real-valued on $K_0$. The intermediate value theorem shows that $1 + \ln(\lambda) \varphi_V(\lambda)$ has at least one zero $\lambda$ on $[0,M]$. It must satisfy
\begin{equation}\label{eq:1p}
\ln(\lambda) = -\dfrac{1}{\varphi_V(\lambda)} = - \dfrac{2\pi}{\epsi^2 \int_{\R^2} \Lambda_0(x) dx + o(\epsi^2)},
\end{equation}
or equivalently,
\begin{equation*}
\lambda = \exp\left( - \dfrac{2\pi}{\epsi^2 \int_{\R^2} \Lambda_0(x) dx + o(\epsi^2) } \right).
\end{equation*}

It remains to show that $\lambda$ is the unique resonance of $V$ in $K_0$ for $\epsi$ small enough. We argue as in \cite[Theorem 2.3]{Si}. If $\mu \neq \lambda$ is another resonance in $K_0$ then $1+\ln(\mu) \varphi_V(\mu) = 0$. Hence,
\begin{equation}\label{eq:1q}
\left| \dfrac{1}{\ln(\mu)} - \dfrac{1}{\ln(\lambda)} \right| = |\varphi_V(\mu)-\varphi_V(\lambda)|.
\end{equation}
In addition, for $\epsi$ small enough, $|\ln(\lambda)|^{-1} \geq c \epsi^2$ and $|\ln(\mu)|^{-1} \geq c \epsi^2$. Hence, a lower bound for the right hand side of \eqref{eq:1q} is
\begin{equation*}
\left| \dfrac{1}{\ln(\mu)} - \dfrac{1}{\ln(\lambda)} \right| \geq \dfrac{1}{|\ln(\lambda)| |\ln(\mu)|} \left|\int_\lambda^\mu \dfrac{dt}{t}\right| \geq \dfrac{|\lambda-\mu|}{|\ln(\lambda)| |\ln(\mu)||\lambda + \mu|} \geq \dfrac{c^2\epsi^4 |\lambda-\mu|}{|\lambda + \mu|}.
\end{equation*}
An upper bound for the left hand side of \eqref{eq:1q} is provided by Lemma \ref{lem:1f}: $|\varphi_V(\lambda)-\varphi_V(\mu)|\leq C|\lambda\ln(\lambda)-\mu \ln(\mu)|$. Therefore, \eqref{eq:1q} implies
\begin{equation*}
\dfrac{c^2 \epsi^4}{|\lambda + \mu|} \leq C \dfrac{|\lambda \ln(\lambda)-\mu \ln(\mu)|}{|\lambda-\mu|}.
\end{equation*}
Since the derivative of $t\ln(t)$ is $1 + \ln(t)$, and since $\lambda, \mu$ remain in a bounded set with $\ln(\lambda) = O(\epsi^{-2})$ (and a similar bound for $\mu$), we deduce that $\epsi^4 \leq C\epsi^{-2}|\lambda + \mu|$. In addition, $|\ln(\lambda)|^{-1} \geq c \epsi^2$ implies $|\lambda| \leq e^{-c^{-1}\epsi^{-2}}$ (and a similar bound for $\mu$). Hence, the existence of $\mu$ implies $c^2\epsi^4 \leq 2C \epsi^{-2} e^{-c^{-1}\epsi^{-2}} = O(\epsi^\infty)$. This is a contradiction.
\end{proof}

\end{document}